\documentclass[10pt]{amsart}
\usepackage[all]{xy}
\usepackage[utf8]{inputenc}
\usepackage{amssymb}
\newcommand{\R}{{\mathbb R} }

\newcommand{\cO}{{\mathcal O} }

\newcommand{\cS}{{\mathcal S} }

\newcommand{\wh}{\widehat}
\newcommand{\wt}{\widetilde}
\newcommand{\pt}{\partial}
\def\ol#1{{\overline{#1}}}
\newtheorem*{Maintheorem*}{Main Theorem}
\newtheorem*{theorem*}{Theorem}

\newtheorem{theorem}{Theorem}
\newtheorem{definition}{Definition}
\newtheorem{lemma}{Lemma}

\newtheorem*{remark*}{Remark}
\newtheorem{proposition}{Proposition}
\newtheorem{corollary}{Corollary}

\def\ke{K{\"a}h\-ler-Ein\-stein }

\def\ka{K{\"a}h\-ler }
\def\he{Her\-mite-Ein\-stein }
\def\wp{Weil-Pe\-ters\-son }

\def\ii{\sqrt{-1}}
\def\ddb{\sqrt{-1}\partial\overline{\partial}}

\def\cinf{C^\infty}

\def\psh{plurisubharmonic}

\title{An extension theorem for hermitian line bundles}
\author[G.~Schumacher]{Georg Schumacher}
\address{Fachbereich Mathematik und Informatik,
Philipps-Universit\"at Marburg, Lahnberge, Hans-Meerwein-Straße, D-35032
Marburg, Germany}
\email{schumac@mathematik.uni-marburg.de}

\begin{document}

\begin{abstract}
We prove a general extension theorem for holomorphic line bundles on reduced complex spaces, equipped with singular hermitian metrics, whose curvature currents can be extended as positive, closed currents. The result has applications to various moduli theoretic situations.
\end{abstract}

\maketitle

\section{Introduction}
Our aim is to prove an extension theorem for holomorphic line bundles $L'$ that are defined on the complement $Y'$ of an analytic subset $A\subset Y$ in a reduced,  complex space of pure dimension $n$ say. We do not have to assume that the given space is projective. However, if we do, an essential step is to solve the extension problem for complex analytic hypersurfaces $H'\subset Y'$. For this situation classical results are known: By the Remmert-Stein theorem \cite{re-st} the closure of $H'$ is an analytic subset of $Y$, if the dimension of $A$ is smaller than $n-1$ everywhere, and Shiffman's theorem \cite{sh1} states that it is sufficient to require that the $2n-3$-dimensional Hausdorff measure of $A$ is equal to zero. More general is Bishops theorem \cite{B}, which only requires that for suitable neighborhoods $U$ of points of $A$ the set $H'\cap (U\backslash A)$ has finite volume.

An analytic aspect of the extension problem for complex hypersurfaces $H'$ is given by the induced current of integration $[H']$, which is a closed positive current of degree $(1,1)$. Here we assume that $Y$ is smooth. Now, if an extension of the current of integration $[H']$ to $Y$ as a positive closed current $T$ is known to exist, then Siu's theorem \cite{siu:lel} implies that the locus, where the Lelong numbers of $T$ are positive, yields an extension of the hypersurface $H'$ to $Y$.

In this sense a natural generalization of the extension problem for analytic hypersurfaces concerns positive, closed $(1,1)$-currents on complex manifolds. Various results are known: Shiffman proved in \cite{shi} that any semi-positive closed $(1,1)$-form of class $\cinf$ can be extended as such into an analytic subset of codimension greater or equal to two. Skoda's theorem \cite{Sk} implies that any positive closed $(1,1)$-current $T'$ on $Y'$ can be extended as such, if it is locally of finite mass near points of $A$. We also mention the extension theorems for positive currents by El Mir \cite{em} and Sibony \cite{si}.

The above reasoning suggests considering the curvature current of the given line bundle that is defined on a reduced not necessarily projective space. We denote by $h'$ a singular hermitian metric on $L'$ over the reduced space $Y'$.

Before turning to singular hermitian metrics and curvature currents, we consider the smooth case: For a complex manifold $Y$ containing an analytic subset $A$ of codimension at least two with complement $Y'$, and a holomorphic line bundle $(L',h')$ on $Y'$ with smooth hermitian metric $h'$ of semi-positive curvature Shiffman shows the extension property in \cite{shi}. The question is about how to treat loci $A$ of codimension one and singularities of $Y$. We show the following theorem.

\begin{theorem}\label{th:main}
Let $Y$ be a reduced complex space and $A\subset Y$ a closed analytic subset. Let $(L',h')$ be a holomorphic line bundle on $Y'= Y\backslash A$ equipped with a singular hermitian metric whose curvature current $\omega_{Y'}= 2 \pi c_1(L',h')$ is positive. We assume that there exists a desingularization of $Y$ such that the pull-back of $\omega_{Y'}$ extends as a positive, closed $(1,1)$-current.

Then there exists a modification $\tau: Z\to Y$, which is an isomorphism over $Y'$ such that $(L',h')$ extends to $Z$ as a holomorphic line bundle equipped with a singular hermitian metric, whose curvature form is a positive current.
\end{theorem}
Even though the hermitian metric $h'$ is being extended, in general the curvature current of the extension is different from the given extension of the curvature current of $h'$.

The above extension theorem applies to moduli theoretic situations, where certain determinant line bundles exist, whose curvature forms are natural \ka forms on the respective moduli spaces. Typical examples are moduli spaces of canonically polarized manifolds (including Riemann surfaces of genus greater than one) and moduli spaces of stable vector bundles on projective varieties. The \ke and \he metrics resp.\ on the fibers induce distinguished \ka metrics on the moduli spaces, namely the \wp metrics and $L^2$ metrics resp. These distinguished \ka forms are known to be equal to the curvature forms of the Quillen metrics on certain determinant line bundles up to a numerical constant  (cf.\ \cite{f-s,sch:inv12}, \cite{vortex,orbi}). Also for classical Hurwitz spaces this construction was carried out \cite{abs}.

On the other hand, compactifications of moduli spaces in the category of complex spaces were known to exist in the category of algebraic spaces (cf.\ \cite{mfk}), and a natural question is to extend these positive holomorphic line bundles to suitable compactifications. In the above mentioned cases we showed explicitly that the \wp\ forms could be extended as positive currents. Our aim is to give a detailed proof of the extension theorem (cf.\ also \cite{sch:inv12}).

We begin with some comments on Theorem~\ref{th:main}. By definition positive currents on reduced complex spaces are currents on the normalization. (If the given metric is of class $\cinf$, positivity of the curvature current corresponds to semi-positivity of the curvature form by definition as usual.) In this sense we will prove the theorem first under the extra hypothesis that $Y$ is normal. If $Y$ is just a reduced complex space, we  pull back the line bundle to the normalization, and extend it.  It is possible to descend the line bundle to a proper modification space over $Z$ (given by a finite map) that is biholomorphic over $Y'$ by Theorem~\ref{th:nonnormal} below.

The first essential step is to treat the smooth divisorial part of $A$ as far as it is contained in the regular locus of $Y$ and use a monodromy argument (cf.\ Proposition~\ref{pr:smoothextgen}). Both, the holomorphic line bundle and the singular hermitian metric are extended into this locus.

Closed, positive $(1,1)$-currents on normal complex spaces in general do not possess $\pt\ol\pt$-potentials (unless these are already given as curvature currents of singular hermitian metrics). Also, potentials are required, in order to use a pull-back of such a current under a holomorphic map. A desingularization of $Y$ is necessary in order to apply the above mentioned methods, and the extension of the curvature current to a desingularization is the weakest possible assumption and most easy to check, hence suitable for moduli theoretic applications.

Descending both the holomorphic line bundle and the singular hermitian metric from a desingularization poses an extra difficulty: The push-forward of the curvature current need not possess a $\pt\ol\pt$-potential and the direct image of the invertible sheaf may only be invertible after applying Hironaka's flattening theorem. This is taken into account in the proof of the main theorem.

\section{Main theorem for complex manifolds}
Grauert's Oka principle states that on a Stein space any topologically trivial holomorphic vector bundle is holomorphically trivial. This fact contains the statement that on non-compact Riemann surfaces all holomorphic line bundles are trivial. We will need the following classical fact.
\begin{lemma}\label{le:deltastar}
All holomorphic line bundles on $\Delta^*\times \Delta^k$ are trivial, where $\Delta$ and $\Delta^*$ resp. denote the unit disk and the punctured unit disk resp.
\end{lemma}
We mention that line bundles on the product of two punctured disks are not necessarily trivial, and Lemma~\ref{le:deltastar} only provides {\em local} extensions of  holomorphic line bundles from $\Delta^*\times \Delta^*$ to $\Delta^2 \backslash \{(0,0)\}$. (Once a line bundle is extended to $\Delta^2 \backslash \{(0,0)\}$ as a trivial line bundle, it can obviously be extended to $\Delta^2$.)

We will consider holomorphic line bundles $L$ equipped with a possibly singular hermitian metric $h$ on a normal complex space. We will call $\ii \Theta(L,h)= 2\pi c_1(L,h)$ the (real) curvature current. Since the curvature current possesses locally a $\pt\ol\pt$-potential that is (locally) integrable, pull-backs of the curvature current under holomorphic maps are well-defined.

Our result will be proven under the further assumption that $X$ is {\em normal}. The general case will be treated in the last section.

Note that on normal complex spaces we use the notion of positivity for $(1,1)$-currents that possess local $\pt\ol\pt$-potentials.

The core of the Theorem is the following version for manifolds and smooth hypersurfaces. We will identify holomorphic line bundles and invertible sheaves and use the same letter for both.

\begin{proposition}\label{pr:smoothext}
  Let $Y$ be a complex manifold and let $D\subset Y$ be a closed analytic smooth hypersurface. Let $(L',h')$ be a holomorphic line bundle on $Y'=Y\backslash D$ with a singular hermitian metric, whose curvature current is positive and extends to $Y$ as a positive closed current $\omega$. Then $(L',h')$ can be extended to $Y$ as a holomorphic line bundle $(L,h)$ equipped with a singular hermitian metric, whose curvature current is positive. The latter current differs from $\omega$ by the sum over currents of integration over the components of $D$ with non-negative coefficients smaller than $2\pi$.
\end{proposition}
  Observe that we do not assume that $D$ is connected.

We state a well-known fact.

\begin{lemma}\label{le:iso}
  Let $Y$ a complex manifold and $L_1$ and $L_2$ holomorphic line bundles on $Y$, equipped with singular hermitian metrics $h_1$ and $h_2$. Let $A\subset Y$ be a closed analytic subset, whose codimension is at least $2$ in any point. Assume that there exists an isomorphism $\alpha: (L_1,h_1)|Y\backslash A \to (L_2,h_2)|Y\backslash A$.
  Then $\alpha$ can be extended to $Y$ in a unique way.
\end{lemma}
\begin{proof}
  The bundle $L_1 \otimes L^{-1}_2$ possesses a nowhere vanishing section over $Y\backslash A$ that can be extended holomorphically to $Y$ with no zeroes. Locally with respect to $Y$ the hermitian metrics $h_j$ are of the form $e^{-\varphi_j}$ for locally integrable functions $\varphi_j$, $j=1,2$, whose difference must be pluriharmonic.
\end{proof}
\begin{corollary}\label{co:snc}
  The statement of  Proposition~\ref{pr:smoothext} holds, if $D$ is a simple normal crossings divisor.
\end{corollary}
\begin{proof}[Proof of Corollary~\ref{co:snc}]
   Let $D=\sum^k_{j=1}D_j$ be the decomposition, where the $D_j$ are the smooth irreducible components of $D$.  Let $Y_\ell= Y \backslash\bigcup^\ell_{j=1}D_j$. We apply Proposition~\ref{pr:smoothext} to $Y_k\subset Y_{k-1}$ and extend $(L',h')$ to $Y_{k-1}$. In the process the current $\omega|Y_{k-1}$ is being replaced by $(\omega+ \gamma _k [D_k])|Y_{k-1}$, $0\leq \gamma_k <2$, which obviously extends to $Y_0=Y$. By induction over $\ell$ we arrive at $\ell=0$, which comprises the claim.
\end{proof}

\begin{proof}[Proof of Proposition~\ref{pr:smoothext}]
  Let $\{U_j \}$ be an open covering of $Y$ such that the set $D\cap U_i$ consists of the zeroes of a holomorphic function $z_i$ on $U_i$. Since all holomorphic line bundles on the product of a polydisk and a punctured disk are trivial by Lemma~\ref{le:deltastar}, we can chose the sets $\{U_i\}$ such that the line bundle $L'$ extends to such $U_j$ as a holomorphic line bundle. So $L'$ possesses nowhere vanishing sections over $U'_j=U_j\backslash D$. Hence $L'$ is given by a cocycle $g_{ij}' \in \cO_Y^*(U_{ij}')$, where $U_{ij}=U_i\cap U_j$ and $U_{ij}'=U_{ij}\cap Y'$. If necessary, we will replace $\{U_i\}$ by a finer covering.

  We fix the notation first. All quantities that carry an apostrophe exist on the complement of the set $A$. Let $e'_i$ be generating sections of $L'$ over $U'_i$ so that for any section $\sigma'$ of $L'$ the equation $\sigma'|U_i= \sigma'_i e'_i$ holds. With transition functions $g'_{ij}$ over $U'_{ij}$ we have
  $$
  e'_j= e'_i g'_{ij} \text{, and } \sigma'_i= g'_{ij} \sigma'_j.
  $$
  The hermitian metric $h'$ is defined on $U'_i$ by the norms of the generators
  $$
  \|e'_i\|^2_{h'} = h'_i \text{ so that } \left(\|\sigma'\|^2_{h'}\right)|U'_i = |\sigma'_i|^2 h'_i.
  $$
  Now over $U'_{ij}$
  $$
  |g'_{ij}|^2 = h'_j/h'_i
  $$
  holds. Later we will change the generating sections of $L'|U'_i$ to $\wt e'_i$, and use the analogous relations.

  The following monodromy argument is essential for the extension of $L'$.

  \par\noindent
  \textbf{Claim.} \textit{ There exist \psh\ functions $\psi_i$ on the open subsets  $U_i \subset Y$ and holomorphic functions $\varphi'_i$ on $U'_i$ such that
  \begin{equation}\label{eq:varphi}
    h'_i\cdot|\varphi'_i|^{-2} = e^{-\psi_i}|U'_i,
  \end{equation}
  where the $\psi_i-\psi_j$ are pluriharmonic functions on $U_{ij}$.}

  \par\noindent
  {\em Proof of the claim. }
  Let
  $$
  \omega|U_i= \ddb (\psi^0_i)
  $$
  for some \psh\ functions $\psi^0_i$ on $U_i$. Now
  $$
  \log(e^{\psi^0_i} h'_i)
  $$
  is pluriharmonic on $U'_i$. For a suitable number $\beta_i\in \R$ and some holomorphic function $f'_i$ on $U'_i$ we have \begin{equation}\label{eq:period}
  \log(e^{\psi^0_i} h'_i) + \beta_i \log|z_i| = f'_i+ \ol{f'_i}.
  \end{equation}
  We write
  $$
  \beta_i=\gamma_i + 2k_i
  $$
  for $0\leq \gamma_i <2$ and some integer $k_i$. We set
  \begin{equation}\label{eq:psi}
  \psi_i = \psi^0_i + \gamma_i \log |z_i|.
  \end{equation}
  These functions are clearly \psh\ on $U_i$, and $\gamma_i \log |z_i|$ contributes as an analytic singularity to $\psi^0_i$. Set
  $$
  \varphi'_i= z^{-k_i}_i e^{f'_i} \in \cO^*(U'_i).
  $$
  With these functions \eqref{eq:varphi} holds. \qed

  The idea of the proof is to introduce new local generators for the given bundle $L'$ over the spaces $U'_i= U_i \backslash A$. This amounts to change the local extensions to the sets $U_i$ of the line bundle $L'$.

  The new generators are
  $$
  \wt e'_i =e'_i {\varphi'_i}^{-1}.
  $$
  In these bundle coordinates the hermitian metric $h'$ on $L'$ is given by
  $$
  \wt h'_i = \|\wt e'_i\|^2= \|e'_i\|^2 |\varphi'_i|^{-2}  = h'_i \cdot |\varphi'_i|^{-2}
  $$
  so that \eqref{eq:varphi} reads
  \begin{equation}\label{eq:wth}
  \wt h'_i = |z_i|^{-\gamma_i} e^{-\psi^0_i}|U'_i= e^{-\psi_i}|U'_i .
  \end{equation}
  The transformed transition functions are
  $$
  \wt g'_{ij}= {\varphi'_i}\cdot g'_{ij}\cdot ({\varphi'_j})^{-1}
  $$
  yielding
  \begin{equation}\label{eq:wtg}
  |\wt g'_{ij}|^2 = \wt h'_j /\wt h'_i = |z_j|^{\gamma_i-\gamma_j} \left|\frac{z_i}{z_j}\right|^{\gamma_i} \cdot e^{\psi^0_i-\psi^0_j}.
  \end{equation}
  Since the function $z_i/z_j$ is holomorphic and nowhere vanishing  on $U_{ij}$, and since the function $\psi^0_i-\psi^0_j$ is pluriharmonic on $U_{ij}$, the function
  $$
  \left|\frac{z_i}{z_j}\right|^{\gamma_i} \cdot e^{\psi^0_i-\psi^0_j}
  $$
  is of class $\cinf$ on all of  $U_{ij}$ with no zeroes.

  Now $-2<\gamma_i-\gamma_j<2$, and $\wt g'_{ij}$ is holomorphic on $U'_{ij}$. So \eqref{eq:wtg} implies that $\wt g'_{ij} \in L^2_{\mathrm{loc}}( U_{ij})$. Hence $\wt g'_{ij}$ can be extended holomorphically to $U_{ij}$, which implies $\gamma_i \geq \gamma_j$, and by symmetry
  \begin{equation}\label{eq:beta}
  \gamma_i=\gamma_j
  \end{equation}
  (whenever $D\cap U_{ij} \neq \emptyset)$. Accordingly \eqref{eq:wtg} now reads
  \begin{equation}\label{eq:wtg1}
  |\wt g'_{ij}|^2 = \wt h'_j/ \wt h'_i = \left|\frac{z_i}{z_j}\right|^{\gamma_i} \cdot e^{\psi^0_i-\psi^0_j}, \end{equation}
  and the transition functions $\wt g'_{ij}$ can be extended holomorphically to all of $U_{ij}$ with no zeroes. So a line bundle $L$ exists.

  The functions $\psi_i$ are \psh, and the quantity
  $$
  \wh \omega = \ddb \psi_i = \omega + {\pi}\,\sum_\nu{\gamma_\nu}\,[D_\nu]
  $$
  is a well-defined positive current on $Y$ because of \eqref{eq:psi} and \eqref{eq:beta}. Its restriction to $Y'$ equals $\omega|Y'$, where the sum is taken over all connected components of $D$.

Now we can define
  $$
  \wt h_i= e^{-\psi_i} = e^{-\psi^0_i}|z_i|^{-\gamma_i}
  $$
  on all of $U_i$. It defines a positive, singular, hermitian metric on $L$. This shows the theorem in the special case.

  Observe that the numbers $\beta_j=\gamma_j+ 2k_j$ only depend on the connected component of the smooth hypersurface $D$.
\end{proof}

The above argument also contains the extension property for lower dimensional sets,  which is already known for metrics of class $\cinf$ (cf.\ Introduction).
\begin{proposition}\label{pr:smoothextcodim2}
  Let $Y$ be a complex manifold and let $A\subset Y$ be a closed analytic subset, whose codimension is at least $2$ everywhere. Let $(L',h')$ be a holomorphic line bundle on $Y'=Y\backslash A$ with a singular hermitian metric, whose curvature form extends as a closed positive current. Then $(L',h')$ can be extended to $Y$ as a holomorphic line bundle $(L,h)$ equipped with a singular hermitian metric in a unique way.
\end{proposition}
Note that the extension property of the curvature current need not be assumed.
\begin{proof}
  By Lemma~\ref{le:iso} we may use a stratification of $A$ and assume that $A$ is smooth of codimension $k\geq 2$. The first step is to produce for any point $a\in  A$ a neighborhood $U(a)$, into which $(L',h')|(U(a)\backslash A)$ can be extended. We write $A\cap U(a)$ as a complete intersection of $n-k$ smooth hypersurfaces, and proceed like in the proof of Corollary~\ref{co:snc}. Because of the assumption on the codimension of $A$ we have trivial monodromy, and for an open covering of $U(a)\backslash A$ the numbers $\beta_i$ from the proof of Proposition~\ref{pr:smoothext} vanish.

  Altogether, we find a covering $\{U_i\}$ of $Y$ such that the line bundle $L'$ is holomorphically trivial on all $U'_i:=U_i\backslash A$. In terms of our standard notation, again the $h'_i=h|U'_i$ can be treated like  functions such that the $-\log h_i$ are plurisubharmonic functions, which can be extended as such to $U_i$ in a unique way. Again the monodromy argument from the proof of Proposition~\ref{pr:smoothext} is not needed -- only the holomorphic functions $f'_i$ can occur, and these are used to change the trivialization over the sets $U'_i$. Finally, the transition functions for $L'$ with respect to $\{U'_i\}$ extend holomorphically (in a unique way) to $\{U_i\}$ (with no zeroes) defining an extension $(L,h)$ of $(L',h')$.
\end{proof}
\begin{proposition}\label{pr:smoothextgen}
  Let $Y$ be a complex manifold and let $A\subset Y$ be a closed analytic subset. Let $(L',h')$ be a holomorphic line bundle on $Y'=Y\backslash A$ with a singular hermitian metric, whose curvature current is positive, and extends to $Y$ as a positive closed current $\omega$. Then $(L',h')$ can be extended to $Y$ as a holomorphic line bundle $(L,h)$ equipped with a singular hermitian metric, whose curvature current is positive. The curvature current of $h$ is equal to $\omega +  \pi\sum \gamma_j [A_j]$, where the $A_j$ are the divisorial components of $A$, and $0\leq \gamma_i <2$.
\end{proposition}

\begin{proof}
  Let $A_1\subset A$ be the divisorial part of $A$, and $A_2$ the union of components in higher codimension. Denote by $A_0$ the singular set of $A_1$. We remove the set $A_0$ and $A_2$ from $Y$ and set $Y''= Y\backslash (A_0 \cup A_2)$. We apply Proposition~\ref{pr:smoothext} and extend $(L',h')$ to $Y''$. Observe that $\omega|Y''$ is altered like in the proof of Corollary~\ref{co:snc}.  We finally apply Proposition~\ref{pr:smoothextcodim2}.
\end{proof}

\section{Main theorem for normal spaces}
\par\noindent\textit{Proof of Theorem~\ref{th:main} for normal spaces.}
  Let $\mu:X \to Y$ be a desingularization of $Y$ such that the pull-back of the curvature tensor of $(L',h')$ extends to a positive closed current $\omega$ on $X$.
  Denote by $Y''\subset Y$ be the regular part. The given current $\omega$ gives rise to a current $\omega''$ on $Y''$. We apply Proposition~\ref{pr:smoothextgen} and extend the line bundle $(L',h')$ to $Y''$ with a resulting bundle $(L'',h'')$. The curvature current of $h''$ equals $\omega''+ \pi\sum \gamma_j [A_j\cap Y'']$, where the $A_j$ are the codimension $1$ components of $A$.  Let $\wt A_j\subset X $ be the proper transforms of the $A_j$ on the manifold $X$. We replace $\omega$ with the positive current $\omega+ \pi \sum \gamma_j [\wt A_j]$, and we are dealing with the line bundle $(L'',h'')$ on $Y''\subset Y$, whose curvature current extends to a desingularization as a positive current.

  Altogether we can assume without loss of generality that the set $A$ is contained in the singular set of $Y$. We proceed with the proof under this additional assumption and use the original notation.

  By Proposition~\ref{pr:smoothextgen} there exists a holomorphic line bundle $  L$ on $X$, equipped with a singular hermitian metric $  h$ whose restriction to $\mu^{-1}(Y')$ is equal to the pull-back of $(L',h')$. We consider the sheaf $\cS=\mu_*(  L)$. Since $Y$ is normal, the fibers of $\mu$ are connected so that $\mu_*(L)|Y'=L'$.

  We apply Hironaka's flattening theorem to the coherent sheaf $\cS$ of rank $1$: There exists a modification $\tau:W\to Y$ given by a  sequence of blow-ups with smooth centers located over $A$ with the property that $\wh L= \tau^* \cS/ \textit{torsion}$ is an invertible sheaf.

\par\noindent
  {\bf Claim 1.} \/ \textit{Without loss of generality we can assume the existence of a holomorphic map $\kappa:X \to W$ such that $\tau\circ \kappa = \mu$.}
  $$
  \xymatrix{& X \ar@{-->}[dl]_\kappa \ar[d]^\mu \\ W \ar[r]^\tau &Y}
  $$
  \begin{proof}[Proof of the Claim]
    The pair $(\mu,\tau)$ defines a bimeromorphic map from $X$ to $W$. Let $\wt W$ be a desingularization of its graph $X\times_Y W$:
    $$
    \xymatrix{\wt W \ar[r]^{\wt \tau}\ar[d]_{\wt\mu}&  X \ar[d]^\mu\\  W \ar[r]^\tau & Y}
    $$
  Now, since the fibers of $\wt\tau$ are connected, $\wt \tau_*\wt\tau^*   L =   L$, and also the curvature current of $\wt \tau^* h$ pushed forward to $X$ equals the original curvature current of $  h$. Hence we may replace $X$ by $\wt W$, the line bundle $  L$ by $\wt \tau^*   L$,  the map $\mu$ by $\mu \circ \wt\tau$, and we set $\kappa= \wt \mu$.
  \end{proof}
  We return to the original situation with the additional map $\kappa$.

  {\bf Claim 2.} \/ \textit{There is an (injective) morphism of invertible $\cO_X$-modules $\kappa^*\wh L \to   L$.}
  \begin{proof}[Proof of Claim 2]
    The canonical morphism $\kappa^*\tau^* \mu_*   L = \mu^*\mu_*   L \to   L$ factors over $\kappa^*(\tau^* \mu_*   L /\textit{torsion}) = \kappa^* \wh L$.
  \end{proof}
  Now we turn to the sheaf $\wh L$. The image of the invertible sheaf $\kappa^* \wh L$ in $  L$ can be identified with $  L(-D)$, where $D$ denotes an effective divisor on $X$, and the injection is the multiplication $\cdot \sigma$ by a canonical section  $\sigma$ of $\cO_X(D)$. It follows from the construction that $D\subset \mu^{-1}(A)$.

  Let $D=0$. We know that $\wh L$ agrees with $L$ over the complement $Y'$ of $A$ in $Y$. Because of Claim~2 the line bundle $L$ is trivial on the fibers of $\kappa$, and since the curvature of $h$ is semi-positive, the singular hermitian metric $h$ descends to $W$ as a (singular) metric $\wh h$ on $\wh L$. The hermitian line bundle $(\wh L, \wh h)$ on $W$ satisfies the claim of Theorem~\ref{th:main}.

  If $D\neq 0$, the hermitian metric on $L$ does not yield a (singular) hermitian metric on $\kappa^*\wh L$ so that zeroes must be taken into account.

  We consider $\tau^*\mu_*(  L(D))$. The multiplication with $\sigma$ induces a morphism $\tau^* \mu_*   L \to \tau^* \mu_* (  L(D))$. Modulo torsion we get
  \begin{equation}\label{eq:Lhatsigma}
  \xymatrix{{\wh L} \ar[r]^-{\cdot \sigma}&\tau^* \mu_* (  L(D))/\textit{torsion}  },
  $$
  and apply $\kappa^*$  (observe that torsion elements are mapped to torsion elements):
  $$
    L(-D)=\xymatrix{{\kappa^*}\wh L \ar[r]^-{\cdot \sigma}\ar[d]_{\cdot \sigma}& (\mu^* \mu_* (  L(D)))/\textit{torsion}\ar[d]^\nu\\   L\ar[r]^{\cdot \sigma}&   L(D)}
  \end{equation}
  where $\nu$ is the natural morphism, which exists, because $L(D)$ is torsion-free. We can see from the diagram that the morphism $\nu$ is induced by the multiplication by $\sigma$.

  We apply the Hironaka flattening process to $\tau^*\mu_*(  L(D))$: There exists a modification $\lambda:Z\to W$ whose exceptional set is mapped to $A$ under $\tau\circ\lambda$ such that
  $$
  \wh{\wh L}=\lambda^*\tau^*\mu_*(  L(D))/\textit{torsion}
  $$
  is an invertible $\cO_Z$-module.

  {\bf Claim 3.} \/ \textit{The invertible sheaf $\wh{\wh L}$ on $Z$, where $\tau\circ\lambda: Z \to Y$ is a modification with center contained in $A$ possesses a singular hermitian metric $\wh{\wh h}$ that satisfies the statement of Theorem ~\ref{th:main}.}
  \begin{proof}[Proof of Claim 3]
  From the definitions of $\wh L$ and $\wh{\wh L}$ we get an injection of invertible sheaves
  $$
  \xymatrix{\lambda^* \wh L \ar[r]^-{\cdot \sigma} &\wh{\wh L}}
  $$
  that is induced by the multiplication with $\sigma$.

  We apply the construction of {\it Claim 1} by letting $V$ be a desingularization of $Z\times_Y X$, and get
  $$
  \xymatrix{V \ar[rr]^\alpha \ar[d]_\rho &   & X \ar[dl]_\kappa \ar[d]^\mu \\ Z \ar[r]^\lambda & W \ar[r]^\tau   & Y}.
  $$
  From \eqref{eq:Lhatsigma} we have a commutative diagram of invertible sheaves on $Z$
  \begin{equation}\label{eq:sheafdia}
  \xymatrix{\rho^* \lambda^* \wh L= \alpha^* \kappa^*( \tau^* \mu_* L/ torsion)
   \ar[r]^-{\cdot \sigma}\ar[d]_{\cdot \sigma} & \rho^*( \lambda^* \tau^* \mu_*(L(D))/torsion = \rho^*\big(\wh{\wh L}\big)\ar[d]^{\cdot \sigma} \\ \alpha^* L \ar[r]^{\cdot \sigma}  & \alpha^* (  L(D)) },
  \end{equation}
  which implies that we can identify the invertible sheaves $\alpha^*   L$ and $\rho^*\big(\wh{\wh L}\big)$ as subsheaves of $\alpha^*(  L(D))$ on $V$.

   We equip $\alpha^* L$ with the pull-back of $h$. Like in the case $D=0$ we see from $\alpha^*L=\rho^*\wh{\wh L}$ that $\alpha^*h$ descends to $\wh{\wh L}$ (since $h$ has semi-positive curvature and since $\rho^*\wh{\wh L}$ is trivial on the fibers of $\rho$). Over $Z\backslash \lambda^{-1}\tau^{-1}(A)$ the hermitian line bundle $(\wh{\wh L}, \wh{\wh h})$ can be identified with $(L',h')$.
   \end{proof}

We mention that by Proposition~\ref{pr:smoothextgen} and Theorem~\ref{th:main} we can trace the  growth of the singular hermitian metric: The given extension of the curvature current of $h'$ is only being changed by adding non-negative multiples of currents of integration located at the boundary with coefficients smaller than $2\pi$, whereas the pullback of the given singular metric $h'$ is just being extended.

\section{Reduced spaces}
Although our extension theorem is based upon analytic methods that are in principle restricted to normal complex spaces, the theorem holds for arbitrary reduced complex spaces.
\begin{definition}
  Let $X$ be a reduced complex space and $L$ a holomorphic line bundle. Then a singular hermitian metric $h$ on $L$ is a singular hermitian metric on the pull-back of $L$ to the normalization of $X$, and its curvature is a current on the normalization of $X$.
\end{definition}

\begin{theorem}\label{th:nonnormal}
Let $Y$ be a reduced complex space, and $A\subset Y$ a closed analytic
subset. Let $L$ be an invertible sheaf on $Y \backslash A$, which
possesses a holomorphic extension to the normalization of\/ $Y$ as an
invertible sheaf. Then there exists a reduced complex space $Z$ together
with a finite map $Z \to Y$, which is an isomorphism over $Y\backslash A$
such that $L$ possesses an extension as an invertible sheaf to $Z$.
\end{theorem}
\begin{proof}
Denote by $\nu:\wh Y \to Y$ the normalization of $Y$. The presheaf
$$
U\mapsto \{\sigma \in  (\nu_* \cO_{\wh Y})(U); \sigma|U\backslash A\in \cO_Y(U\backslash A)\}
$$
defines a coherent $\cO_Y$-module, the so-called  {\em gap sheaf}
$$
\cO_Y[A]_{\nu_*\cO_{\wh Y}}
$$
on $Y$ (cf.\ \cite[Proposition 2]{siu:gap}). It carries the structure of
an $\cO_Y$-algebra. According to Houzel \cite[Prop.\ 5 and Prop.\ 2]{hou}
it follows that it is an $\cO_Y$-algebra of finite presentation, and hence
its analytic spectrum provides a complex space $Z$ over $Y$ (cf.\ also
Forster \cite[Satz 1]{fo}).
\end{proof}

\end{document}